\documentclass{article}

\usepackage[utf8]{inputenc}
\usepackage[french,english]{babel}
\usepackage{amsmath,amsfonts}
\usepackage{bbm}
\usepackage{color} 
\usepackage{txfonts}
\usepackage{amsfonts,enumitem,dsfont}
\usepackage{amssymb}
\usepackage{hyperref} 
\usepackage{graphicx,tikz,todonotes,comment}

\usepackage{algorithmic}
\usepackage[ruled,noend]{algorithm2e}
\usepackage[skins]{tcolorbox}

\newcommand{\N}{\mathbb{N}}

\newcommand{\HH}{\mathcal{H}}
\newcommand{\CC}{\mathbb{C}}

\newcommand{\R}{\mathbb{R}}

\newtheorem{thm}{Theorem}
\newtheorem{lemma}{Lemma}

\newtheorem{definition}{Definition}
\newtheorem{ass}{Assumption}

\newenvironment{proof}[1][]{\noindent {\bf Proof #1:\;}}{\hfill $\Box$}

\def\<{\langle}
\def\>{\rangle}

\renewcommand{\d}{\partial}

\begin{document}

\author{ Ludovick Gagnon $^{1}$}
\footnotetext[1]{CNRS, Inria, F-54000 Nancy, France. E-mail : \texttt{ludovick.gagnon@inria.fr},}

\title{Note on the Exponential Stability of the Viscoelastic Wave Equation}

\maketitle

\begin{abstract}

In this article, we study the viscoelastic wave equation with a localized history memory term, and without frictional damping. Under the geometric control condition on the memory region, we prove exponential decay of the viscoelastic energy using the Gearhart-Pr\"uss-Huang Theorem. We prove in high frequency that a loss of derivative is compensated by a decay in frequency.    \\

\textup{2020} \it MSC: 35Q93, 35A27, 35L05, 35L51, 35S15.

\end{abstract}

\section{Introduction}


Let $\Omega \subset \R^d$ be a bounded and connected domain. We assume $\d \Omega $ to be of class $C^3$. Let $\R^+=\{x\in \R \, | \, x>0\}$. We consider the viscoelastic equation,
\begin{equation}\label{MP1}
\begin{cases}
u_{tt}(x,t) - \Delta u(x,t) \displaystyle + \int_{-\infty}^t g(t-s) \textrm{ div}(a(x) \nabla u(x,s))\,ds =0, & (x,t)\in \Omega \times \R^+ \\
u(x,t)=0, & (x,t)\in \d \Omega \times \R^+ \\
u(x,0)=u_0(x),~u_t(x,0)=u_1(x),& x\in \Omega
\end{cases}
\end{equation}
where the initial data $(u_0,u_1)$ belongs to $H^1_0(\Omega) \times L^2(\Omega)$ and the past history of $u$ is defined by 
\[
u(x,-s)=\phi_0(x,s), \quad s\in \R^+.
\]
It is classical to reformulate the non-autonomous equation \eqref{MP1} into an autonomous one with the change of variables introduced by Dafermos \cite{Dafermos}
\begin{eqnarray}\label{change}
\eta(x,t,s) := u(x,t) - u(x,t-s), \quad (x,t,s)\in \overline{\Omega}\times [0,\infty) \times  [0,\infty).
\end{eqnarray}
Indeed, 
\begin{align*}
\int_{-\infty}^t g(t-s) \textrm{div}(a(x) \nabla u(x,s))\,ds&= \int_0^{\infty}g(s) \textrm{div} (a(x) \nabla u(x,t-s))\,ds \\
&= \int_0^{\infty}g(s) \textrm{div} \big(a(x) \nabla \left[ u(x,t)-\eta(x,t,s)\right]\big) \,ds. 
\end{align*}
Thus, assuming $g\in L^1(\R^+)$ and introducing  
\[
g_0:=\int_0^\infty g(s) \,ds, \quad 
\kappa(x) := 1-g_0 a(x),
\]
the augmented autonomous system, equivalent to \eqref{MP1}, is given by,
\begin{equation}\label{AUP}
\begin{cases}
\displaystyle u_{tt}(x,t) -  \textrm{div}( \kappa(x) \nabla u(x,t)) -  \int_0^\infty g(s) \textrm{div}(a(x)\nabla \eta(x,t,s))\,ds=0, \hfill (x,t)\in \Omega \times \R^+\\
\eta_t(x,t,s) + \eta_s(x,t,s) =u_t(x,t), \hfill  (x,t,s)\in  \Omega \times \R^+ \times \R^+\\ 
u(x,t)=\eta(x,t,s)=0, \hfill  (x,t,s)\in \d \Omega\times \R^+ \times \R^+ \\   
u(x,0)=u_0(x), ~u_t(x,0)=u_1(x), \hfill  ~x\in \Omega \, \, \,\\ 
\eta(x,t,0)=0,~\eta(x,0,s):=\eta_0(x,s)=u_0(x)-\phi_0(x,s), \hfill  ~(x,t,s)\in \Omega \times \R^+ \times 
\R^+ \end{cases}
\end{equation}
where the past history of $u$ satisfies,
\[
\eta_0(x,s)=u_0(x)-\phi_0(x,s)\in L^2(\R^+;H^1_a(\Omega)),
\]
with $L^2(\R^+;H^1_a(\Omega))$ defined in Section \ref{sec:FF}. For $t\geq 0$, the \textit{viscoelastic energy} of (\ref{AUP}) is defined by,
\begin{align*}
E(t) & := \dfrac{1}{2} \int_{\Omega} |u_t(x,t)|^2 + \kappa(x) |\nabla u(x,t)|^2 dx   + \dfrac{1}{2}  \int_{\Omega} \int_0^\infty g(s) a(x) |\nabla \eta(x,t,s)|^2 ds dx.
\end{align*}
The energy is non-increasing 
\begin{align*}\label{energy-identity}
\dfrac{d}{dt}E(t) = \dfrac12 \int_0^{+\infty} g'(s)\, \int_{\Omega} a(x)|\nabla \eta(x,t,s) |^2\,dxds \leq 0,
\end{align*}
assuming $a$ nonnegative and $g$ is non-increasing.
\subsection{Main results}

We begin by introducing classical hypothesis on   \eqref{MP1} and \eqref{AUP}. 

\begin{ass}\label{assump:pa} 
We assume 
\begin{itemize}
\item the function $a\in C(\overline\Omega)$ is nonnegative. We define the open set 
\begin{equation}\label{def:omega}
\omega := \big\{ x\in \Omega \, | \, a(x)>0\big\};
\end{equation}
\item the function $g\in W^{1,1}(0,\infty) \cap C^1(0,\infty)$ is a positive, non-increasing function satisfying, for $c>0$, 
\[
g(s) \leq -cg'(s), \quad \forall s\in \R^+;
\]
\item we assume $a(x)$ and $g(s)$ satisfy, for $C>0$, 
\[
\kappa(x)=1-g_0a(x)\geq C, \quad \forall x\in \Omega,
\]
\end{itemize}
\end{ass}

We establish the exponential stability of \eqref{AUP} under GCC \cite{BLR}, 
\begin{definition}[Geometric Control Condition (GCC)]
The pair $(\omega,\Omega)$ satisfies the geometric control condition in time $T>T^*$ if and only if every generalized bicharacteristic associated to the symbol 
\[
p(x,t,\tau,\xi) = \tau^2 - |\xi|^2, \quad (x,t,\tau,\xi) \in T_b^*(\Omega\times \R^+) 
\]
and following the laws of geometric optics enters non-diffractively the observation region $\omega$ in a time $T>T^*$.
\end{definition}
Our main result is the following 
\begin{thm}\label{thm:main}
Let $T>0$ and suppose Assumption \ref{assump:pa} holds. Assume $ \omega$ satisfies GCC in time $T$. Then, exponential stability holds, that is there exists $c,C>0$ such that
\[
E(t)\leq Ce^{-ct}E(0), \quad \forall t \geq T.
\]
\end{thm}
The equivalence between exponential stability and the observability inequality is classical \cite{CoursBG} for autonomous systems. By means of Theorem \ref{thm:main}, we therefore also establish the following observability inequality under GCC for \eqref{MP1}. 

\begin{lemma}
Let $T>0$ and suppose Assumption \ref{assump:pa} holds. Assume $\omega$ satisfies GCC for $T>0$. Then, there exists $c_T>0$ such that the following observability inequality
\begin{eqnarray}\label{Obs ineq}\,\,\qquad
E(0)\leq c_T
\left(\int_0^T\int_{0}^{\infty}(-g'(s)) \int_{\Omega}  a(x)|\nabla \eta(x,t,s)|^2 \, dxdsdt\right).
\end{eqnarray}
holds for $T>T^*$. 
\end{lemma}

\subsection{Strategy of the proof}

Theorem \ref{thm:main} is proven using Gearhart-Pr\"uss-Huang Theorem (see \cite[Theorem 1.3.2]{LiuZheng1999} and also \cite{Wyler1994}). 
\begin{thm}[Gearhart-Pr\"uss-Huang Theorem]
Let $S(t)=e^{At}$ be a $C_0$-semigroup of contractions on a Hilbert space. Then $S(t)$ is exponentially stable if and only if 
\begin{equation}\label{Ires}
i\R \subset \rho(A)
\end{equation}
and  
\begin{equation}\label{supres}
\sup_{\mu \in \R } \|(A- i\mu I )^{-1}\| < \infty
\end{equation}
hold.
\end{thm}

The proof of Theorem \ref{thm:main} is done by arguing by contradiction on Gearhart-Pr\"uss-Huang Theorem. The main difficulty in the proof of Theorem \ref{thm:main} is the following. Assume \eqref{supres} does not hold. Consider a renormalized sequence $\| U_n\|=1$ such that 
\begin{equation}\label{intro:res}
(A-i \mu_n I)U_n = F_n.
\end{equation}
with $F_n \to 0$ as $n\to \infty$ in the state space (introduced in Section \ref{sec:FF}). Contradiction is achieved if $U_n \to 0$ as $ n\to \infty$ in the state space. The delicate analysis is at high frequency $|\mu_n|\rightarrow \infty$. An equivalent formulation of this problem is by solving
\begin{equation}\label{intro:wavesm}
(\Delta + \mu_n^2) w_n = g_n - \textrm{div} J_n, \quad \textrm{in } H^{-1}(\Omega), 
\end{equation}
together with the strong convergence $\| \sqrt{a}(x) \nabla w_n(x)\|_{L^2(\omega;\CC^d)} \rightarrow 0$ as $n\to \infty$. The standard contradiction argument consists to propagate this strong convergence from the observability region to the whole domain $\Omega$. This may be done by recasting \eqref{intro:wavesm} in the semiclassical analysis framework, 
\begin{equation}\label{intro:semiclassical}
(h^2 \Delta + 1) y_h =h \left( g_h -  \textrm{div} J_h \right), \quad \textrm{in } H^{-1}_{sc}(\Omega), 
\end{equation}
with $hy_h=w_n$, where $h$ scales as $1/\mu_n$. A $L^2$ semiclassical measure may then be attached to $y_h$ to study this propagation. The $L^2$ propagation theorems for the $(h^2 \Delta + 1)$ operator require a $o(h)$ decay of the right-hand side of \eqref{intro:semiclassical} in $L^2$-norm. This result is proved for $\Omega=\R^d$ by Zworski \cite[Theorem 5.3 and 5.4]{ZworskiBook} and by Burq  \cite[Section 4]{BurqNontrapping} for non-trapping obstacles in $\R^d$. This latter result is summarized in Theorem \ref{thm:burq-boundary-propagation} and is used in the present article.

The loss of derivative in the right-hand side of \eqref{intro:wavesm} prevents \textit{a priori} the direct application of Theorem \ref{thm:burq-boundary-propagation}. The core of our analysis is to solve explicitly \eqref{intro:res}, and to prove that the loss of derivative of $\textrm{div} J_n$ can be compensated using its explicit expression. More specifically, the regularity of $g(s)$ transfers to frequency decay that leads to,
\[
\lim_{n\to \infty} |\mu_n| \|J_n \|_{L^2(\Omega)} = 0.
\]
Decomposing $y_h$ as $\chi(-h^2\Delta) y_h + (1-\chi(-h^2\Delta) ) y_h $, where $\chi$ is a microlocalization in frequency near $1$, allows us to apply Theorem \ref{thm:burq-boundary-propagation} for the regular $\chi(-h^2\Delta) y_h$ component, while standard elliptic regularity allows us to handle $(1-\chi(-h^2\Delta) ) y_h $. This translates to $U_n \to 0$ as $n \to \infty$, the desired contradiction. 

The rest of Theorem \ref{thm:main} consists to analyse the low frequency regime and to analyse the resolvent of $A$. Care has to be taken with the latter part since $A$ does not necessarily have a compact resolvent.

\subsection{Bibliographical review}

There exists an extensive literature on the stability of wave equations with viscoelastic damping, both with a prescribed past history and with memory integrals - that is, with $\phi_0(x,t)\equiv 0$. We refer the reader to the literature reviews in \cite{Cavalcanti2,Jorge}. Here we restrict our discussion to localized memory and to the propagation argument relevant to the present note.\newline

The results in the literature on the energy decay for a memory term including a past history also include an additional frictional damping mechanism,
\begin{equation}\label{memory+damp}
u_{tt}(x,t) - \Delta u(x,t) \displaystyle + \int_{-\infty}^t g(t-s) \textrm{ div}(a(x) \nabla u(x,s))\,ds + b(x) u_t(x,t)=0,
\end{equation}
with the central hypothesis $a(x)+b(x) \geq C>0$ in $\Omega$. To the best of our knowledge, the first paper dealing with the case $b\equiv 0$ and $a$ satisfying some geometrical control conditions is \cite{Cavalcanti2}. The strategy consists to propagate strong convergence of $u$ over the observability region to the whole domain $\Omega$. In their proof, \eqref{AUP} is recasted as a single wave equation, with strong convergence over the observability region, but without intermediate energy estimates between $\eta$ and $u$. The authors argue directly on the solution $u$, and more specifically on $v=u_t$. Taking the time derivative of the wave equation results in a loss of derivative of the right-hand side of \cite[equations~(3.106)--(3.107)]{Cavalcanti2} as the right-hand side belongs to $H^1((0,T);H^1(\Omega))' \subset H^{-2}((0,T) \times \Omega)$. The classical propagation results for the defect measure in the energy space only apply with a right-hand side in $ H^{-1}((0,T) \times \Omega)$. The same strategy is adopted in \cite{Jorge, Piupiu}. The energy decay is therefore not immediate from these proofs in the case $b\equiv 0$. \newline

Another approach in the literature consists to add an auxiliary damping $j^{-1}\beta(x)\partial_tu_j$ and then letting $j\to\infty$ (see \cite{FariaOkawaSouzaFranco2026, FariaSouzaFranco2023}). Passing to the limit requires an observability estimate whose constant is independent of $j$. However, it is not clear from the proofs that this constant remains bounded as the auxiliary damping vanishes. The conclusion for the unperturbed problem therefore appears to require an additional argument.\newline

The purpose of this paper is to clarify the propagation argument for \eqref{AUP}. Working directly with the resolvent, we show that the loss of derivative is compensated by decay in frequency, which yields strong convergence in the energy space. This gives a complete proof of exponential stability without frictional damping. \newline

The Gearhart-Pr\"uss-Huang Theorem allow us to find an explicit expression between $u$ and $\eta$ in the contradiction argument to exhibit the required frequency decay. This theorem was used for instance by Liu and Rao \cite{LiuRao2004} for wave equations with localized Kelvin-Voigt damping, by Augner and Jacob \cite{AugnerJacob2014} for boundary stabilization of port-Hamiltonian systems and by Conti, Dell'Oro and Pata \cite{ContiDellOroPata2013} for a viscoelastic beam model of Timoshenko type.

\subsection{Structure of the paper} 

In Section \ref{sec:FF}, we recall the functional framework associated to \eqref{AUP}. We modify slightly the functional framework presented in \cite{Cavalcanti2} to gain the dissipativity of the infinitesimal generator $A$ of the contraction semigroup associated to \eqref{AUP}. Indeed, in \cite{Cavalcanti2}, the well-posedness is proved by showing that $A$ is the infinitesimal generator of a quasicontraction semigroup, that is, there exists $\tilde{\omega}\geq 1/2$ such that, 
\[
\| e^{At} \| \leq e^{\tilde{\omega} t}, \quad t\geq 0.
\]
This is not in contradiction with the exponential stability as, proving the observability inequality holds in time $T>T^*$ implies that the exponential bound holds with $\tilde{\omega}<0$ for $t>T^*$. The quasicontraction semigroup would however force us to use Pr\"uss's results on $C_0$-semigroup on Hilbert spaces \cite{Pruss1984}
\[
 \sigma(e^{At}) \setminus\{ 0\} \subset e^{\sigma(A)t}
\]
on the strip $\{z \in \CC \, | \, 0\leq \textrm{Re}(z) \leq \tilde{\omega} \}$, instead of on the resolvent set along the imaginary axis. Proving the dissipativity therefore simplifies the proof of Theorem \ref{thm:main}. \newline

\noindent In Section \ref{sec:WP}, we use Lumer-Phillips to show that $A$ is indeed the infinitesimal generator of a contraction semigroup. We then use one of its variants to recall that $A$ is closed and of dense domain. \newline 

\noindent We prove Theorem \ref{thm:main} in Section \ref{sec:main} by arguing by contradiction on the Gearhart-Pr\"uss-Huang Theorem. We prove a high frequency observability inequality allowing us to recover the loss of derivative, at the cost of a frequency weight (see Lemma \ref{lem:hfobs}). \newline

\noindent The proof of Lemma \ref{lem:hfobs} is done in Section \ref{sec:semiclassic}. It relies on the $L^2$ semiclassical defect measure introduced by Burq in \cite{BurqNontrapping} for nontrapping obstacles, recalled in Theorem \ref{thm:burq-boundary-propagation}. \newline

\noindent Section \ref{sec:lemmatech} is devoted to the proof of technical lemmas.

\section{Functional framework}\label{sec:FF}

We endow $L^2(\Omega)$ with its usual scalar product and norm. Consider the space $H^1_0(\Omega)$ endowed with the weighted scalar product, 
\[
(u,v)_{H^1_0(\Omega)}=\int_\Omega \kappa(x) \nabla u(x) \cdot \nabla v(x) \, dx
\]
and the $H^1_0$-norm,
\[
\|u \|_{H^1_0(\Omega)}^2= \int_\Omega \kappa(x)  | \nabla u(x) |^2 dx. 
\]
Assumption \ref{assump:pa} ensures $\kappa(x)=1-g_0a(x)\geq C>0$ on $\Omega$ implying that this norm is equivalent to the classical $H^1_0(\Omega)$-norm, hence the abuse of notation. Consider the space for the memory term,
\[
L^2(\R^+;H^1_a(\Omega)):=\left\{ \eta(.,s) :s\in \R^+ \rightarrow H^1_0(\Omega)  \, \, \Bigg| \, \int_\Omega \int_0^\infty a(x) g(s) | \nabla \eta(x,s)|^2 dsdx < \infty \right\},
\]
and define
\[
\|\eta \|_{L^2(\R^+;H^1_a(\Omega))}^2 := \int_\Omega \int_0^\infty a(x) g(s) | \nabla \eta(x,s)|^2 dsdx.
\]
Notice that $\|. \|_{L^2(\R^+;H^1_a(\Omega))}$ may only define a semi-norm on $L^2(\R^+;H^1_a(\Omega))$ since 
\[
\|u\|_{H^1_a(\Omega)}^2 := \int_\Omega  a(x)  | \nabla u(x)|^2 dx
\]
is only a semi-norm on $H^1_0(\Omega)$ if $a$ vanishes in $\Omega$. The space $H^1_a(\Omega)$ is understood as the completion of the quotient space $H^1_0(\Omega)$ with respect to the kernel of the semi-norm $\|.\|_{H^1_a(\Omega)}^2 $. The space $L^2(\R^+;H^1_a(\Omega))$ is then the weighted Bochner space with the measure $g(s)ds$. Its scalar product is given by, 
\[
(\eta, \xi)_{L^2(\R^+;H^1_a(\Omega))} := \int_\Omega \int_0^\infty a(x) g(s) \nabla \eta(x,s) \cdot \nabla \xi(x,s) \, dsdx.
\]
Finally, we introduce the state space 
\[
\HH:=H^1_0(\Omega) \times L^2(\Omega) \times L^2(\R^+;H^1_a(\Omega))
\]   
endowed by the norm 
\[
\| (u,v,\eta) \|_{\HH}^2:= \|u\|_{H^1_0}^2 + \|v\|^2_{L^2} +  \|\eta\|_{L^2(H^1_a)}^2
\]
where we drop the notations of $\Omega$ and $\R^+$ to lighten the notations. The scalar product on $\HH$ is given by 
\[
\langle .,. \rangle_{\HH}:= (.,.)_{L^2} +  (.,.)_{H^1_0} +  (.,.)_{L^2(H^1_a)}
\]
and one easily verifies that $\HH$ is a reflexive Hilbert space.  Define $A:D(A)\subset \HH \rightarrow \HH$ by 
\begin{equation}\label{def:A}
A \begin{pmatrix} u \\ v\\ \eta \end{pmatrix}= \begin{pmatrix} v\\   \textrm{div}( \kappa(x) \nabla u(x))+ \displaystyle \int_0^\infty g(s) \textrm{div}(a(x) \nabla \eta(x,s)) \, ds  \\ v-\eta_s \end{pmatrix}
\end{equation}
with its domain,
\begin{align}\label{def:DA}
D(A):=   \Bigg\{ \, & (u,v,\eta) \in \HH  \, \, \Big|  \ \, v\in H^1_0(\Omega), \quad \eta_s \in L^2(\R^+;H^1_a), \quad \eta(x,0)=0, \nonumber \\
&    \qquad   \left(\textrm{div} ( \kappa(x) \nabla u(x)) + \displaystyle \int_0^\infty g(s) \textrm{div}(a(x) \nabla \eta(x,s)) ds\right) \in L^2(\Omega) \Bigg\}.
\end{align}
In the next section, Lemma \ref{lemm} proves that $A$ is closed with dense domain $D(A)$ in $\HH$. 

\section{Well-posedness}\label{sec:WP}

We propose a different proof of the well-posedness of \cite{Cavalcanti2} to recover the dissipativity of the operator $A$. Let us recall 
\begin{thm}[Lumer-Phillips Theorem (Corollary 3.20, \cite{EN})]
Let $A :D(A) \subset \HH \rightarrow \HH $ be a linear operator, and $\HH$ a reflexive Banach space. Then $A$ is the infinitesimal generator of a $C_0$-semigroup of contractions if and only if 
\begin{itemize}
\item $A$ is dissipative;
\item $A-\lambda_0 I$ is surjective for some $\lambda_0>0$.
\end{itemize}
\end{thm}
We have, 

\begin{thm}
Let $A:D(A) \subset \HH \rightarrow \HH$ be defined by \eqref{def:A}-\eqref{def:DA}. Then $A$ is the infinitesimal generator of a $C_0$ semigroup of contractions on $\HH$. 
\end{thm}

\begin{proof}
The surjectivity is proven in \cite[Theorem 2.1]{Cavalcanti2} with $\lambda_0=1+\tilde{\omega}$, for $\tilde{\omega}\geq 1/2$. We turn to the dissipativity of $A$. Let  $U = (u,v,\eta) \in D(A)$. We have, 
\[
\langle AU, U \rangle_{\mathcal{H}} = I_1 + I_2 + I_3
\]
with 
\begin{align*}
I_1 &= \int_{\Omega}\kappa(x)  \nabla v(x) \cdot \nabla u(x) \, dx \\
I_2 &= \int_{\Omega}  \left( \textrm{div} (\kappa(x) \nabla u(x)) + \int_0^\infty g(s) \text{div}(a(x) \nabla \eta(x,s)) \, ds \right) v(x)  \, dx \\
I_3 &= \int_{\Omega} \int_0^\infty g(s) a(x) \nabla (v(x) - \eta_s(x,s)) \cdot \nabla \eta(x,s) \, ds \, dx.
\end{align*}
Integrating by parts and using $v\in H^1_0(\Omega)$, 
\[
I_2 = -\int_{\Omega} \kappa(x)\nabla u(x) \cdot \nabla v(x) \, dx - \int_{\Omega} \left( \int_0^\infty g(s) a(x) \nabla \eta(x,s) \cdot \nabla v(x) \, ds \right) dx.
\] 
Thus, 
\[
I_1+I_2= -\int_{\Omega} \left( \int_0^\infty g(s) a(x) \nabla \eta(x,s) \cdot \nabla v(x) \, ds \right) dx.
\]
Moreover, we have 
\[
I_3 =  \int_{\Omega} \int_0^\infty g(s) a(x) \nabla v(x) \cdot \nabla \eta(x,s) \, ds \, dx - \int_{\Omega} \int_0^\infty g(s) a(x) \nabla \eta_s(x,s) \cdot \nabla \eta(x,s) \, ds \, dx
\]
which, in turn, gives,
\[
I_1+I_2+I_3 = - \int_{\Omega} \int_0^\infty g(s) a(x) \nabla \eta_s(x,s) \cdot \nabla \eta(x,s) \, ds \, dx.
\]
Now, since $\eta(x,0)=0$ in $H^1_a(\Omega)$, we have $\sqrt a\nabla \eta(x,0)=0$ in $L^2(\Omega)$. Moreover, since $\eta \in L^2(\R^+;H^1_a(\Omega))$, we have, 
\[
\lim_{s \rightarrow \infty} g(s)\int_\Omega a(x)|\nabla \eta(x,s)|^2 dx = 0.
\] 
Integrating by parts on $(0,S)$ for $S>0$ and then letting $S \rightarrow\infty$ yields
\[
 - \int_{\Omega} \int_0^\infty g(s) a(x) \nabla \eta_s(x,s) \cdot \nabla \eta(x,s) \, ds \, dx = \dfrac{1}{2} \int_{\Omega} \int_0^\infty g'(s) a(x) |\nabla \eta(x,s)|^2 \, ds \, dx. 
\]
The right-hand side of the equation above is negative since $g'(s)<0$ as $g(s)$ is non-increasing. Hence, 
\[
\langle AU, U \rangle_{\mathcal{H}} \leq 0,
\]
which is sufficient to prove the dissipativity of $A$, 
\[
\langle (\lambda I-A) U, U \rangle_{\mathcal{H}} \geq \lambda \|U\|_{\mathcal{H}}^2
\]
for any $\lambda>0$. 
\end{proof}
\newline

We recall a variant of Lumers-Phillips Theorem
\begin{thm}
Let $A$ be a linear operator defined on a linear subspace $D(A)$ of the Banach space $X$. Then $A$ generates a quasi contraction semigroup if and only if
\begin{itemize}
\item $D(A)$ is dense in $X$;
\item $A$ is closed;
\item $A$ is quasidissipative, i.e. there exists a $\tilde{\omega}\geq 0$ such that $A - \tilde{\omega}I$ is dissipative;
\item $A - \lambda_0 I$ is surjective for some $\lambda_0 > \tilde{\omega}$, where I denotes the identity operator.
\end{itemize}
\end{thm}
to recover
\begin{lemma}\label{lemm}
The operator $A$ defined by \eqref{def:A}-\eqref{def:DA} is closed, dissipative and has a dense domain $D(A)$ in $\HH$. 
\end{lemma} 
Lemma \ref{lemm} will be used in the next section to prove that  $A$ is closed and densely defined, which will be used in the next section.

\section{Exponential decay under GCC}\label{sec:main}

We now turn to the proof of Theorem \ref{thm:main}. The main argument is presented in this section, while some technical lemmas and the proof of the high frequency observability inequality are postponed to the next two sections. \newline

\begin{proof}
From Gearhart-Pr\"uss-Huang Theorem, the exponential decay is equivalent to 
\begin{equation}\label{proof:ires}
i\R \subset \rho(A)
\end{equation}
and
\begin{equation}\label{proof:sup}
\sup_{\mu \in \R } \|(i\mu I -A )^{-1}\| < \infty. 
\end{equation}
The strategy of the proof is the following. We first prove the uniform estimate
\begin{equation}\label{uniformresest}
\|U\|_{\HH}\leq C\|(i\mu I-A)U\|_{\HH},
\quad \mu\in\R,\quad U\in D(A).
\end{equation}
with $C>0$ independent of $\mu$. We then prove that $i\R\subset\rho(A)$. The uniform estimate \eqref{uniformresest} then gives \eqref{proof:sup} since the norm on the resolvent is the subordinate norm. \newline

\textbf{Step 1: Uniform estimate on the imaginary axis :}\newline

We argue by contradiction. Assume that there exists a sequence $\mu_n \in \R$ and $U_n=(u_n,v_n,\eta_n)\in D(A)$ such that $\|U_n\|_{\HH}=1$ and 
\[
(A-i \mu_n I )U_n = F_n \rightarrow 0, \quad \textrm{ in } \HH,
\]
where we denote $F_n=(f_n,g_n,h_n) \in \HH$. Up to extraction on the sequence $\mu_n$, either $|\mu_n|\rightarrow\infty$ or $\mu_n\rightarrow\mu\in\R$. Using the definition of $A$, 
\begin{equation}\label{sequence}
\begin{cases}
v_n- i\mu _n u_n=f_n \rightarrow 0, \hfill \text{   in   }H_0^1(\Omega ) \\ 
\textrm{div}(\kappa(x) \nabla u_n) +  \displaystyle \int _{0}^{\infty } g(s) \text{div}(a(x) \nabla \eta_n (s)) \,ds - i\mu _n v_n=g_n \rightarrow 0, \hfill \text{   in   }L^{2}(\Omega )\\
v_n - \partial _{s }\eta_n-i\mu _n \eta_n=h_n \rightarrow 0, \hfill \text{   in   }L^{2}(\mathbb{R}^{+};H_{a}^{1}(\Omega ))\end{cases}
\end{equation}
Using the dissipativity computations yield
\[
\operatorname{Re}\langle F_n,U_n\rangle_{\HH}
=\operatorname{Re}\langle AU_n,U_n\rangle_{\HH}
=\dfrac{1}{2} \int_0^\infty \int_\Omega g'(s)a(x) |\nabla \eta_n(x,s) |^2 dx ds.
\]
Since $\|U_n\|_{\HH}=1$ and $F_n\rightarrow0$ in $\HH$, we deduce
\[
\left|\operatorname{Re}\langle AU_n,U_n\rangle_{\HH}\right|
\leq \|U_n\|_{\HH} \|F_n\|_{\HH} \leq \|F_n\|_{\HH}  \rightarrow0,
\]
which leads to,
\begin{equation}\label{gradetantendvers0}
\int_0^\infty \int_\Omega g'(s)a(x) |\nabla \eta_n(x,s) |^2 dx ds \rightarrow 0,  \quad n\rightarrow \infty. 
\end{equation}
From $g(s) \leq -c g'(s)$, we recover, 
\begin{equation*}
\int_0^\infty \int_\Omega g(s)a(x) |\nabla \eta_n(x,s) |^2 dx ds \rightarrow 0,  \quad n\rightarrow \infty,
\end{equation*}
which is precisely, 
\begin{equation}\label{convetaforte}
\lim_{n\rightarrow \infty} \| \eta_n \|_{L^2(H^1_a)} =0.
\end{equation} 

We first give an equivalent formulation of  \eqref{sequence}. Since $\kappa(x)=1-g_0a(x)$, we have  
\begin{equation}\label{eq:wavehf}
\Delta u_n + \mu_n^2 u_n = i\mu_n f_n + g_n - \textrm{div} H_n, \quad \textrm{in } H^{-1}(\Omega),
\end{equation}
where 
\[
H_n(x) := \int_0^\infty g(s) a(x) \nabla( \eta_n(x,s) - u_n(x)) \, ds  \quad \textrm{in } L^2(\Omega;\CC^d).
\]
We underline that \eqref{eq:wavehf} only holds in $H^{-1}(\Omega)$ since, for $(u_n,v_n,\eta_n)\in D(A)$, we only have $u_n\in H^1_0(\Omega)$ and 
\[
\textrm{div}( \kappa(x) \nabla u_n(x)) + \displaystyle \int_0^\infty g(s) \textrm{div}(a(x) \nabla \eta_n(x,s)) \, ds \in L^2(\Omega)
\]
while each term of the sum belongs to $H^{-1}(\Omega)$. Second, we solve explicitly the third line of \eqref{sequence} using the first line of \eqref{sequence} and $\eta_n(.,0)=0$ since $\eta_n \in D(A)$, 
\begin{equation}\label{relwvh}
\eta_n(x,s) = (1-e^{-i\mu_n s})u_n(x)+ \left(\dfrac{1-e^{-i\mu_n s}}{i\mu_n}\right)f_n(x)-\int_0^s e^{i\mu_n (\sigma-s) }h_n(x,\sigma) d\sigma,  
\end{equation}
This equation holds in $ C(\R^+;H^1_a(\Omega))$ thanks to the regularity in the $s$-variable of the right-hand side. This solution motivates the change of variables
\begin{equation}\label{def:wn}
w_n = u_n + \dfrac{f_n}{i\mu_n }.
\end{equation}
We write \eqref{eq:wavehf} in the $w_n$ variable, 
\begin{equation}\label{eq:wavew}
\Delta w_n + \mu_n^2 w_n = g_n - \textrm{div} J_n, \quad \textrm{in } H^{-1}(\Omega),
\end{equation}
with 
\begin{equation}\label{def:Jn}
J_n := H_n -  \dfrac{\nabla f_n}{i\mu_n }, \quad \textrm{in } L^2(\Omega;\CC^d). 
\end{equation}
We also write \eqref{relwvh} as, 
\begin{equation}\label{relwn}
\eta_n(x,s) = (1-e^{-i\mu_n s})w_n(x)-(\mathcal{K}_{\mu_n}h_n)(s) , \quad  \forall n \in \N. 
\end{equation}
where, to ease the notation, we denote 
\begin{equation}\label{defknu}
(\mathcal{K}_{\mu_n}h)(s) := \int_0^s e^{i\mu_n (\sigma-s) }h(x,\sigma) d\sigma.
\end{equation}
Two main ingredients are used to conclude the contradiction argument: translate the observability in $\eta_n$ to $w_n$ and propagate the observability in $\omega$ to $\Omega$. The following technical lemma - whose proof is given in Section \ref{sec:lemmatech} - establishes the observability in the $w_n$ variable, as well as a high frequency estimate for $J_n$. 
\begin{lemma}\label{convaunabla}
Let $w_n$ be defined by \eqref{def:wn} and $J_n$ by \eqref{def:Jn}. Assume $\mu_n \rightarrow \mu^* \in ( \R \cup \{ \pm \infty \} ) \setminus \{0\}$. Then, 
\begin{equation}\label{wnobs}
\lim_{n\rightarrow \infty} \big\| \sqrt{a(.)}\nabla w_n \big\|_{L^2(\Omega;\CC^d)} = 0,
\end{equation}
and 
\begin{equation}\label{Jnconv}
\lim_{n\rightarrow \infty} |\mu_n | \, \| J_n \|_{L^2(\Omega;\CC^d)} = 0.
\end{equation}
\end{lemma}

\noindent The rest of the proof of Step 1 is divided in two regimes. \newline 

\textbf{High Frequency:}\newline

\noindent We begin with the case $|\mu_n| \rightarrow \infty$.

\begin{lemma}\label{lem:hfobs}
There exist $C>0$ and $\mu^*>0$ such that, for any $|\mu|\geq \mu^*$, every $w\in H^1_0(\Omega)$, $g\in L^2(\Omega)$ and $J\in L^2(\Omega;\CC^d)$ satisfying 
\[
(\Delta + \mu^2 ) w = g - \textrm{div} J, \quad \textrm{in } H^{-1}(\Omega),
\]
satisfies also the high frequency observability inequality,
\begin{equation}\label{obs:hf}
\| \nabla w \|^2_{L^2} + \mu^2 \|  w \|^2_{L^2} \leq C \left( \int_{\Omega} a(x) | \nabla w(x) |^2 dx + \|g\|_{L^2}^2 + \mu^2 \| J \|_{L^2(\Omega;\CC^d)}^2 \right).
\end{equation}
\end{lemma}
In Lemma \ref{lem:hfobs}, the loss of derivative of $J$ is compensated by a decay with respect of the parameter $\mu$. Its proof is found in Section \ref{sec:semiclassic}.\newline

We use Lemma \ref{convaunabla} and $g_n \rightarrow 0 $ in $L^2(\Omega)$ in Lemma \ref{lem:hfobs} to deduce that the right-hand side of \eqref{obs:hf} tends to zero as $n\rightarrow \infty$. We thus obtain, 
\begin{equation}\label{strconvw}
\lim_{n \rightarrow \infty } \left( \|\nabla w_n \|^2_{L^2} + \mu_n^2 \| w_n \|^2_{L^2} \right)= 0.
\end{equation}
Since $f_n \rightarrow 0$ in $H^1_0(\Omega)$, we have by the definition \eqref{def:wn} of $w_n$  that 
\[
\lim_{n\rightarrow \infty } \| \nabla u_n \|_{L^2(\Omega)} =0 
\]
Moreover, since $f_n \rightarrow 0$ in $H^1_0(\Omega)$, we deduce by Poincaré's inequality $f_n \rightarrow 0$ in $L^2(\Omega)$, which implies, together with \eqref{strconvw}, 
\[
|\mu_n|\|u_n\|_{L^2}
\leq |\mu_n|\|w_n\|_{L^2}+\|f_n\|_{L^2}
\longrightarrow0,
\]
that is, 
\begin{equation}\label{L2unconv}
\lim_{n\rightarrow \infty } \mu_n^2 \|u_n \|_{L^2(\Omega)}^2 =0 
\end{equation}
Using \eqref{L2unconv} and the first equation of \eqref{sequence}, we conclude 
\begin{equation*}
\lim_{n\rightarrow \infty } \|v_n \|_{L^2(\Omega)}^2 =0 
\end{equation*}
Together with \eqref{convetaforte}, we conclude 
\[
\lim_{n\rightarrow \infty} \| (u_n,v_n,\eta_n) \|_{\HH} =0
\]
a contradiction with $ \| (u_n,v_n,\eta_n) \|_{\HH} =1$.   \newline

\textbf{Low frequency :}\newline
 
\noindent We consider now the case $\mu_n \rightarrow \mu \in \R$. Recall that \eqref{convetaforte} gives $\eta_n\rightarrow0$ in $L^2(\R^+;H^1_a(\Omega))$. We distinguish two cases. First, if $\mu=0$, then the first equation of \eqref{sequence} gives $v_n=f_n\to 0$ in $H^1_0(\Omega)$. The second equation of  \eqref{sequence}, using the boundedness of $u_n$ in $H^1_0(\Omega)$, yields, 
\[
\int_\Omega\kappa(x)|\nabla u_n|^2\,dx \rightarrow0.
\]
Since this is an equivalent norm of $H^1_0(\Omega)$, we recover $U_n \to 0$ in $\HH$, a contradiction. We thus turn to the case $\mu \neq 0$. Since $u_n$ is bounded in $H^1_0(\Omega)$, we deduce that there exists a subsequence such that 
\[
u_n \rightharpoonup u \quad \text{in } H^1_0(\Omega),\qquad u_n \rightarrow u \quad \text{in } L^2(\Omega).
\]
From Lemma \ref{convaunabla}, we have, 
\[
\|\sqrt a\nabla w_n\|_{L^2(\Omega;\CC^d)}\rightarrow 0.
\]
Combined with $f_n/i\mu_n \rightarrow 0$ in $H^1_0(\Omega)$, we obtain
\[
\sqrt{a(x)} \nabla u(x)=0\quad\textrm{in }L^2(\Omega;\CC^d).
\]
Passing to the limit in the second equation of \eqref{sequence} gives
\[
\operatorname{div}(\kappa (x) \nabla u)+\mu^2u=0, \quad \text{ in } H^{-1}(\Omega).
\]
Since $a(x) \nabla u=0$ and $\kappa=1-g_0a(x)$, this reduces to
\begin{equation}\label{Deltalimite}
\Delta u+\mu^2u=0,
\end{equation}
which holds in $C^\infty(\Omega)$ from elliptic regularity. On the open set $\omega=\{a>0\}$, one has $\nabla u=0$, hence $u$ is constant over $\omega$. But using \eqref{Deltalimite} over $\omega$ and $\mu\neq0$, one deduces $u=0$ on $\omega$. Finally, the unique-continuation for the operator $(\Delta+\mu^2)$ gives $u=0$ in $C^\infty(\Omega)$. This implies $v_n \rightarrow 0$ in $L^2(\Omega)$, and together with \eqref{convetaforte}, we recover $U_n \rightarrow 0$ in $\HH$. This therefore concludes the proof of Step 1. \newline

\textbf{Step 2: the resolvent set includes the imaginary axis}\newline

\noindent The estimate proved in Step~1 shows, for every $\mu\in\R$, that $i\mu I-A$ is injective and has closed range. It remains to prove that this range is dense. We prove 
\[
\operatorname{Ran}(i\mu I-A)^\perp
=\operatorname{Ker}(-i\mu I-A^*)=\{0\}.
\]
Let us first detail the computations for the adjoint. Let $U=(u,v,\eta)\in D(A)$ and let $V=(p,q,\zeta)$ be smooth enough. From the definition of $A$, we have
\[
\begin{aligned}
\langle AU,V\rangle_{\HH}
={}&\int_\Omega \kappa(x) \nabla v(x) \cdot \nabla p(x) \,dx\\
&+\int_\Omega\left(\operatorname{div}(\kappa(x) \nabla u(x) )
+\int_0^\infty g(s)\operatorname{div}(a(x)\nabla \eta(x,s))\,ds\right)
 q\,dx\\
&+\int_0^\infty\int_\Omega g(s)a(x)
\nabla(v(x)-\eta_s(x,s))\cdot\nabla \zeta(x,s)\,dxds.
\end{aligned}
\]
Integrations by part in $x$ and using $q\in H_0^1(\Omega)$ gives
\[
\begin{aligned}
\langle AU,V\rangle_{\HH}
={}&-\int_\Omega\kappa(x) \nabla u(x) \cdot\nabla q(x)\,dx
-\int_\Omega v(x) \textrm{div}( \kappa(x) \nabla p(x)) \,dx \\
&-\int_0^\infty\int_\Omega g(s)a(x) \nabla \eta(x,s)
\cdot\nabla  q(x) \,dxds \\
& -\int_0^\infty\int_\Omega g(s)v(x)
\textrm{div}( a(x) \nabla \zeta(x,s) )\,dxds\\
&-\int_0^\infty\int_\Omega g(s)a(x)\nabla \eta_s(x,s)
\cdot\nabla \zeta(x,s) \,dxds.
\end{aligned}
\]
For the last term, integrating by parts in $s$ and using $\eta(x,0)=0$, 
\[
\begin{aligned}
-\int_0^\infty\int_\Omega g(s)a(x) & \, \nabla \eta_s(x,s)
\cdot\nabla \zeta(x,s) \,dxds \\
&=\int_0^\infty\int_\Omega a(x)\nabla \eta(x,s)
\cdot  \left(g(s) \nabla \zeta_s(x,s)+g'(s) \nabla \zeta(x,s) \right)\,dxds \\
&=\int_0^\infty\int_\Omega g(s) a(x)\nabla \eta(x,s)
\cdot  \nabla \left(  \zeta_s(x,s)+\dfrac{g'(s)}{g(s)}  \zeta(x,s) \right)\,dxds 
\end{aligned}
\]
since $g(s)>0, s\geq 0$. We finally gather the terms associated to the scalar product of $\HH$, 
\[
\begin{aligned}
\langle AU,V\rangle_{\HH}
={}& - \int_\Omega\kappa\nabla u(x)\cdot \nabla q(x)\,dx\\
&-\int_\Omega v(x) \operatorname{div}\left(
\kappa(x) \nabla p(x)+\int_0^\infty g(s)a(x) \nabla \zeta(x,s)\,ds\right)  
\,dx \\
&+\int_0^\infty\int_\Omega g(s) a(x)\nabla \eta(x,s)
\cdot \nabla \left( -q(x)+  \zeta_s(x,s)+\dfrac{g'(s)}{g(s)}\zeta(x,s) \right)\,dxds 
\end{aligned}
\]
Consequently,
\[
A^*\begin{pmatrix}p\\q\\\zeta\end{pmatrix}
=\begin{pmatrix}
-q\\
-\operatorname{div}( \kappa(x) \nabla p(x))-\displaystyle\int_0^\infty g(s) \operatorname{div}( a(x) \nabla \zeta(x,s))\,ds \\
-q(x)+\zeta_s(x,s)+\dfrac{g'(s)}{g(s)}\zeta(x,s)
\end{pmatrix}
\]
with domain
\[
\begin{aligned}
D(A^*)=\Bigg\{(p,q,\zeta)\in\HH\ \Bigg|\;&q\in H_0^1(\Omega),\quad
\zeta_s+\dfrac{g'(s)}{g(s)}\zeta\in L^2(\R^+;H_a^1(\Omega)),\\
& \left( \operatorname{div}( \kappa(x) \nabla p(x))+\displaystyle\int_0^\infty g(s) \operatorname{div}( a(x) \nabla \zeta(x,s))\,ds \right) \in L^2(\Omega)
\Bigg\}.
\end{aligned}
\]
In contrast with the domain of $A$, no boundary condition at $s=0$ is imposed on $\zeta$. Now, let $V=(p,q,\zeta)\in D(A^*)$ satisfying $A^*V=-i\mu V$, 
that is, 
\begin{equation}\label{vpadjoint}
\begin{pmatrix}
-q\\
-\operatorname{div}( \kappa(x) \nabla p(x))-\displaystyle\int_0^\infty g(s) \operatorname{div}( a(x) \nabla \zeta(x,s))\,ds \\
-q(x)+\zeta_s(x,s)+\dfrac{g'(s)}{g(s)}\zeta(x,s)
\end{pmatrix}
= -i \mu \begin{pmatrix}p\\q\\\zeta\end{pmatrix} 
\end{equation}
Arguing with the real part of the scalar product over $\HH$, 
\[
\operatorname{Re}\langle A^*V,V\rangle_{\HH}
=\dfrac12\int_0^\infty\int_\Omega g'(s)a(x)|\nabla \zeta(x,s)|^2\,dxds
-\dfrac{g(0)}2\int_\Omega a(x)|\nabla \zeta(x,0)|^2\,dx.
\]
Since $\operatorname{Re}\langle A^*V,V\rangle_{\HH} =\operatorname{Re}( -i\mu \|V\|_{\HH}^2)=0$, this computation above is well-defined. Moreover, we deduce 
\begin{equation}\label{zeroegalzero}
\dfrac12\int_0^\infty\int_\Omega g'(s)a(x)|\nabla \zeta(x,s)|^2\,dxds = \dfrac{g(0)}2\int_\Omega a(x)|\nabla \zeta(x,0)|^2\,dx.
\end{equation}
But by assumption, $g(0)>0$ and $g'(s)<0$ since $g$ is non-increasing. Thus, the right-hand side of the previous equation is nonnegative while the left-hand side is nonpositive, which forces both terms to be equal to zero. In particular, using the assumption $g(s)\leq-cg'(s)$
\begin{align}
0 &=-c \int_0^\infty\int_\Omega g'(s)a(x)|\nabla \zeta(x,s)|^2\,dxds  \label{derivzeta0} \\
 &\geq \int_0^\infty\int_\Omega g(s)a(x)|\nabla \zeta(x,s)|^2\,dxds \nonumber \\
  &= \|\zeta\|_{L^2(H^1_a)}^2 \label{zetanorm0}
\end{align}
Thus, since $g(s)$ is real-valued, we deduce from the third line of \eqref{vpadjoint},
\begin{align*}
\| \zeta_s - q\|_{L^2(H^1_a)}^2 & = \int_0^\infty \int_\Omega  g(s)  a(x) \left| \dfrac{g'(s)}{g(s)} + i\mu \right|^2 \left| \nabla \zeta(x,s) \right|^2 dx ds \\
& \leq 2 \mu^2 \int_0^\infty   \int_\Omega g(s) a(x)  \left| \nabla \zeta(x,s) \right|^2 dx ds \\ 
& \quad + \dfrac{2}{c} \int_0^\infty  \int_\Omega   g'(s)a(x)   \left| \nabla \zeta(x,s) \right|^2 dx ds \\
&=0
\end{align*}
using $g\leq-cg'$, \eqref{derivzeta0} and \eqref{zetanorm0}. Thus, 
\[
\zeta_s(x,s)=q(x), \quad \text{in } L^2(\R^+;H^1_a(\Omega)).
\]
We finally obtain,
\begin{align*}
\int_0^\infty \int_\Omega g(s) a(x) | \nabla q(x) |^2 dx ds  =&  \int_0^\infty \int_\Omega g(s) a(x) \nabla \zeta_s(x,s) \cdot \nabla \overline{ q(x) }  dx ds  \\
=& - \int_0^\infty \int_\Omega g'(s) a(x) \nabla  \zeta(x,s) \cdot \nabla \overline{ q(x) } dx ds \\
&  - g(0) \int_\Omega a(x)  \nabla  \zeta(x,0) \cdot \nabla \overline{ q(x) }  dx 
\end{align*}
and using Cauchy-Schwarz together with the fact that \eqref{zeroegalzero} is equal to zero, 
\begin{align*}
\Bigg|  \int_0^\infty & \int_\Omega g(s) a(x) | \nabla q(x) |^2 dx ds \Bigg|^2 \\ \leq &  \left( \int_0^\infty \int_\Omega g'(s) a(x) | \nabla  \zeta(x,s) |^2  dx ds \right)   \left( \int_0^\infty \int_\Omega g'(s) a(x) | \nabla  q(x) |^2  dx ds \right) \\
& + \left( g(0) \int_\Omega  a(x) | \nabla  \zeta(x,0) |^2  dx  \right) \left(  \int_\Omega  a(x) | \nabla  q(x) |^2  dx  \right)  \\
=& \, 0.
\end{align*}
Using the first two equations of \eqref{vpadjoint} finally yields, 
\[
\Delta q + \mu^2 q=0,
\qquad
\sqrt a\nabla q=0.
\]
Arguing again by unique continuation allows us to conclude $V=0$, and consequently $\operatorname{Ker}(-i\mu I-A^*)=\{0\}$. The range of $i\mu I-A$ is therefore both closed and dense, hence equal to $\HH$. Thus $i\R\subset\rho(A)$. Moreover, the uniform estimate of Step~1 yields
\[
\sup_{\mu\in\R}\|(i\mu I-A)^{-1}\|_{\mathcal L(\HH)}<\infty.
\]
This concludes the contradiction argument, allowing to prove the exponential stability using Gearhart-Pr\"uss-Huang's Theorem.
\end{proof}

\section{High frequency observability inequality}\label{sec:semiclassic}

We now turn to the proof of Lemma \ref{lem:hfobs}. We use semiclassical measures up to the boundary in the sense of \cite[Section~4.2]{BurqNontrapping}. We denote
\[
\Sigma_b
=
j\bigl(\{(x,\xi)\in T^*\overline{\Omega};\ |\xi|^2=1\}\bigr)
\]
the compressed characteristic set. Here $j$ is the canonical map near the boundary: in local coordinates $(x',x_n)\in \R^{d-1}\times \R$ with $\Omega =\{x_n >0\}$, we have $j(x,\xi)=(x',x_n,\xi', x_n \xi_n)$ where $\xi=(\xi',\xi_n)\in \R^{d-1}\times \R$. We refer the reader to \cite{timedependentwave} for a presentation of the compressed cotangent bundle $T_b^{*}(\overline{\Omega})$ of Melrose-Sj{\"o}strand \cite{MelroseSjostrand1978, MelroseSjostrand1982}. 

We recall that in \cite{BurqNontrapping}, the pseudodifferential operators $A$ near the boundary are a sum of interior semiclassical pseudodifferential operators and tangential semiclassical pseudodifferential operators in local coordinates with compactly supported symbols. The notation $\sigma(A)$ denotes the semiclassical symbol of $A$. 
\begin{thm}\label{thm:burq-boundary-propagation}
Let $\Omega$ be a bounded domain of $\R^d$ with boundary of
class $C^3$. Let
\[
v_h\in H^2(\Omega)\cap H^1_0(\Omega),
\qquad
\sup_h\|v_h\|_{L^2(\Omega)}<+\infty,
\]
and assume that
\[
\|(h^2\Delta+1)v_h\|_{L^2(\Omega)}=o(h).
\]
After extraction of a subsequence, there exists a positive
Radon measure $\nu$ on $\Sigma_b$ such that
\[
\lim_{h\to0}(A_hv_h,v_h)_{L^2(\Omega)}
=
\int_{\Sigma_b}\sigma(A_h)\,d\nu.
\]

The support of $\nu$ is a union of generalized
bicharacteristics. Finally, $\nu$ does not charge the diffractive set $\mathcal{G}^{2,+}$. 
\end{thm}
\begin{proof}
The setting of \cite[Section~4]{BurqNontrapping} applies directly here. In \cite[Section~4]{BurqNontrapping}, $\Theta \subset \R^d$ is a smooth compact and nontrapping obstacle whose complement is $\Omega= \R^d \setminus \Theta$. Homogeneous Dirichlet boundary conditions are assumed on $\d \Omega$, and the propagation of the $L^2$ defect measure is studied on $ \|(h^2\Delta+1)v_h\|_{L^2_{loc}(\Omega)}=o(h)$. Here, the bounded geometry of $\Omega$ simplifies slightly the presentation of \cite[Section~4]{BurqNontrapping}. The $C^3$ regularity of the boundary may be handled with \cite[Theorem 6.1]{BDLRI} and \cite[Theorem 3.4]{BDLRII}.
\end{proof}

\begin{lemma} \label{lem:measure-vanishes-on-omega}
Let $(v_h)$ and $\nu$ be as in
Theorem~\ref{thm:burq-boundary-propagation}. Assume that $a\in C(\overline\Omega)$ is nonnegative and
\[
\int_\Omega a(x)|h\nabla v_h|^2\,dx \rightarrow0.
\]
Then
\[
\nu\big|_{\pi^{-1}(\omega)}=0,
\qquad
\omega=\{x\in\Omega;\ a(x)>0\}.
\]
\end{lemma}

\begin{proof}
Let $\chi\in C_c^\infty(\omega)$ be real-valued and nonnegative. Since $a$ is continuous and
strictly positive on the compact set $\operatorname{supp}\chi=\overline{ \{ x\in \omega \, | \, \chi(x)>0\} }$, there
exists $c_\chi>0$ such that
\[
a(x)\geq c_\chi
\quad\text{on }\operatorname{supp}\chi.
\]
Consequently,
\[
\|\chi h\nabla v_h\|_{L^2(\Omega)}^2
\leq
\frac{\|\chi\|_{L^\infty}^2}{c_\chi}
\int_\Omega a(x)|h\nabla v_h|^2\,dx
\rightarrow 0.
\]
Passing to the limit and using the definition of the semiclassical measure,
\[
\int_{\Sigma_b}\chi(x)^2|\xi|^2\,d\nu(x,\xi)=0.
\]
But since $|\xi|^2=1$ on $\Sigma_b$, it follows that
\[
\int_{\Sigma_b}\chi(x)^2\,d\nu(x,\xi)=0.
\]
As $\chi\in C_c^\infty(\omega)$ is arbitrary, $\nu$ vanishes above
$\omega$.
\end{proof}\newline

We now turn to the proof of Lemma \ref{lem:hfobs}.\newline 

\begin{proof}
Suppose that \eqref{obs:hf} does not hold and denote $h=1/|\mu|$. Then, there exist $h\rightarrow0$ and $(w_h,g_h,J_h)$ satisfying
\begin{equation}
\label{eq:hf-normalization}
\|\nabla w_h\|_{L^2}^2+h^{-2}\|w_h\|_{L^2}^2=1,
\end{equation}
\begin{equation}
\label{eq:hf-equation}
(\Delta+h^{-2})w_h=g_h-\textrm{div}J_h
\quad\textrm{in }H^{-1}(\Omega),
\end{equation}
and
\begin{equation}
\label{eq:hf-small-rhs}
\int_\Omega a(x)|\nabla w_h|^2dx
+\|g_h\|_{L^2}^2+h^{-2}\|J_h\|_{L^2(\Omega;\CC^d)}^2
\rightarrow0.
\end{equation}
Define
\[
y_h=h^{-1}w_h.
\]
Then \eqref{eq:hf-normalization} becomes for the semiclassical $H^1$-norm,
\begin{equation}
\label{eq:yh-normalization}
\|y_h\|_{H_{sc}^1}^2
:=\|y_h\|_{L^2}^2+\|h\nabla y_h\|_{L^2}^2=1.
\end{equation}
Multiplying \eqref{eq:hf-equation} by $h$ gives
\begin{equation}
\label{eq:semiclassical-equation}
P_hy_h:=\left(h^2\Delta +1\right)y_h
=h g_h-h\textrm{div} J_h=:F_h
\quad\textrm{in }H_h^{-1}.
\end{equation}
For every $\phi\in H_0^1(\Omega)$,
\begin{align*}
\left|\left\langle h\textrm{div}J_h,\phi\right\rangle\right|
&=h\left|\int_\Omega J_h\cdot\nabla \overline\phi\,dx\right|\\
&\leq\|J_h\|_{L^2(\Omega;\CC^d)}\|h\nabla \phi\|_{L^2}
\leq\|J_h\|_{L^2(\Omega;\CC^d)}\|\phi\|_{H_{sc}^1}.
\end{align*}
Consequently, \eqref{eq:hf-small-rhs} implies
\begin{equation}
\label{eq:Fh-small}
\|F_h\|_{H_{sc}^{-1}}
\leq h\|g_h\|_{L^2}+\|J_h\|_{L^2(\Omega;\CC^d)}
=o(h).
\end{equation}
Moreover,
\begin{equation}
\label{eq:hf-observation}
\int_\Omega a(x)|h\nabla y_h|^2dx
=\int_\Omega a(x)|\nabla w_h|^2dx \rightarrow0.
\end{equation}

Let $0<\delta<1/4$ and choose $\chi\in C_c^\infty((0,\infty))$ such that
\[
\chi(s)=1\quad\textrm{if }|s-1|\leq\delta,
\qquad
\textrm{supp}\,\chi\subset\{s\in(0,\infty):|s-1|<2\delta\}.
\]
Using the functional calculus for the (positive) Dirichlet Laplacian $-\Delta$, set
\[
v_h=\chi(-h^2\Delta)y_h.
\]
Since $\chi(-h^2\Delta)$ and $P_h$ commute, we have
\begin{equation}
\label{eq:vh-equation}
P_hv_h=\chi(-h^2\Delta)F_h.
\end{equation}
Moreover, since the compact support of $\chi$ localizes functions in frequency, we deduce that $\chi(-h^2\Delta)F_h$ belongs to $D( (-h^2\Delta)^k)$ for $k\in \N$. The embedding and \eqref{eq:Fh-small} therefore yield, 
\begin{equation}
\label{eq:localized-rhs}
\|P_hv_h\|_{L^2}
\leq C\|F_h\|_{H_{sc}^{-1}}=o(h).
\end{equation}
On the other hand, the operator $(h^2\Delta +1)$ is elliptic for $(1-\chi(-h^2\Delta))y_h$. Thus,  
\begin{equation}
\label{eq:elliptic-complement}
\|(1-\chi(-h^2\Delta ))y_h\|_{H_{sc}^1}
\leq C\|F_h\|_{H_{sc}^{-1}}=o(h).
\end{equation} 
Now, since $\|v_h\|_{L^2}$ is bounded, we can attach an $L^2$ semiclassical measure to it thanks to Theorem \ref{thm:burq-boundary-propagation}. Moreover, from \eqref{eq:hf-observation} and \eqref{eq:elliptic-complement}, we have, 
\[
\int_\Omega a(x)|h\nabla v_h|^2dx \rightarrow 0.
\]
Lemma \ref{lem:measure-vanishes-on-omega} therefore shows that $\nu$ vanishes above the observation region $\omega=\{a>0\}$. The geometric control condition and Theorem \ref{thm:burq-boundary-propagation} then imply $\nu=0$, that is, 
\[
\|v_h\|_{L^2(\Omega)}\rightarrow0.
\]
But since the support of $\chi$ is included in $\{ s\in (0,\infty) \, | \, |s-1|< 2\delta \}$, we conclude there exist uniform bounds 
\[
c \|v_h \|_{L^2(\Omega)} \leq \|h \nabla  v_h \|_{L^2(\Omega)}  \leq C\|v_h \|_{L^2(\Omega)}  \rightarrow 0. 
\]
which implies $\|v_h\|_{H^1_{sc}} \rightarrow 0$ as $h\rightarrow \infty$. 
Together with \eqref{eq:elliptic-complement}, this yields
\[
\|y_h\|_{H_{sc}^1}\longrightarrow0,
\]
which contradicts \eqref{eq:yh-normalization}. This proves Lemma \ref{lem:hfobs}.
\end{proof}

\section{Technical Lemmas}\label{sec:lemmatech}

This section is devoted to the proof of Lemma \ref{convaunabla}. First, denote by $g(s)$ the extension of $g$ by zero for $s<0$ and set
\[
\hat{g}(\tau) = \int_{-\infty}^{\infty} g(s)e^{-i \tau s} ds, \quad \tau \in \R.
\]
\subsection{Preliminary technical lemma}
The proof of Lemma \ref{convaunabla} relies on the following technical lemma. 
\begin{lemma}\label{lemtech}
Assuming $\mu_n \neq 0 $, we have, 
\begin{equation}\label{estimghat}
|\hat{g}(\mu_n) | \leq \dfrac{2g(0)}{|\mu_n|},
\end{equation}
\begin{equation}\label{estimKhn}
\|\mathcal{K}_{\mu_n}h_n\|_{L^2(H^1_a)}
\leq 2C\|h_n\|_{L^2(H^1_a)}.
\end{equation}
\begin{equation}\label{estimhn}
\left\| \int_0^\infty g(s) 
\int_0^s e^{i\mu_n (\sigma-s) } a(x) \nabla h_n(x,\sigma) d\sigma ds \right\|_{L^2(\Omega;\CC^d)} \leq \dfrac{C}{|\mu_n|} \|h_n\|_{L^2(H^1_a)},
\end{equation}
for $C>0$ independent of $n\in \N$.
\end{lemma}

\begin{proof}
Since $g$ is non-increasing, positive and integrable, we have
\[
\lim_{s\rightarrow\infty}g(s)=0,
\qquad
\int_0^\infty |g'(s)|ds=g(0).
\]
An integration by parts gives, for $\mu_n\neq0$,
\[
\hat g(\mu_n)
=\frac{1}{i\mu_n}\left(g(0)+\int_0^\infty g'(s)e^{-i\mu_ns}ds\right).
\]
The bound \eqref{estimghat} follows directly,
\[
|\hat g(\mu_n)|
\leq\frac{g(0)+\|g'\|_{L^1(\R^+)}}{|\mu_n|}
=\frac{2g(0)}{|\mu_n|}.
\]

We next estimate $\mathcal K_{\mu_n}h_n$,
\[
(\mathcal{K}_{\mu_n}h_n)(s) := \int_0^s e^{i\mu_n (\sigma-s) }h_n(x,\sigma) d\sigma.
\]
From its definition,
\[
\partial_s(\mathcal K_{\mu_n}h_n)(s)
=h_n(.,s)-i\mu_n(\mathcal K_{\mu_n}h_n)(s),
\qquad
(\mathcal K_{\mu_n}h_n)(0)=0.
\]
For $S>0$, integration by parts and $g(s)\leq-cg'(s)$ give
\begin{align*}
\int_0^S\int_\Omega g(s)a(x)
|\nabla \mathcal K_{\mu_n}h_n|^2dxds
&\leq-c\int_0^S\int_\Omega g'(s)a(x)
|\nabla \mathcal K_{\mu_n}h_n|^2dxds \\
& \leq 2c \textrm{Re}\int_0^S\int_\Omega g(s)a(x)
\nabla h_n\cdot\nabla \overline{\mathcal K_{\mu_n}h_n}
dx ds \\ 
& \quad +2c \textrm{Re} i\mu_n \int_0^S\int_\Omega  g(s)a(x)  |\nabla \mathcal K_{\mu_n}h_n |^2
dx ds  \\ 
& \quad - c \int_\Omega g(S)a(x) | \nabla \mathcal K_{\mu_n}h_n (S) |^2
dx \\
&\leq 2c\,\textrm{Re}\int_0^S\int_\Omega g(s)a(x)
\nabla h_n\cdot\nabla \overline{\mathcal K_{\mu_n}h_n}
dx ds
\end{align*}
since the boundary term is nonpositive and the real part of one of the integrant is zero. We use the Cauchy-Schwarz inequality on the right-hand side of the inequality to recover,
\[
\left(\int_0^S\int_\Omega g(s)a(x)
|\nabla \mathcal K_{\mu_n}h_n|^2dxds\right)^{1/2}
\leq 2c
\left(\int_0^S\int_\Omega g(s)a(x)
|\nabla h_n|^2dxds\right)^{1/2}.
\]
Letting $S\rightarrow\infty$, we obtain
\[
\|\mathcal K_{\mu_n}h_n\|_{L^2(H^1_a)}
\leq2c\|h_n\|_{L^2(H^1_a)}.
\]

It remains to prove \eqref{estimhn}. Fubini's theorem gives
\begin{align*}
&\int_0^\infty g(s)\int_0^s
e^{i\mu_n(\sigma-s)}a(x)\nabla h_n(x,\sigma)d\sigma ds\\
&\qquad=\int_0^\infty
\left(\int_\sigma^\infty g(s)e^{i\mu_n(\sigma-s)}ds\right)
a(x)\nabla h_n(x,\sigma)d\sigma.
\end{align*}
Thus, using \eqref{estimghat} and Cauchy-Schwarz, 
\begin{align*}
\Bigg\|\int_0^\infty g(s) \int_0^s &
e^{i\mu_n(\sigma-s)}a(x) \nabla h_n(x,\sigma)d\sigma ds
\Bigg\|_{L^2(\Omega;\CC^d)}\\
&\quad\leq\frac{2}{|\mu_n|}\int_0^\infty g(\sigma)
\|a\nabla h_n(\sigma)\|_{L^2(\Omega;\CC^d)}d\sigma\\
&\quad\leq\frac{2\|a\|_{L^\infty(\Omega)}^{1/2}}{|\mu_n|}
\int_0^\infty g(\sigma)
\|\sqrt a\nabla h_n(\sigma)\|_{L^2(\Omega;\CC^d)}d\sigma\\
&\quad\leq\frac{2g_0^{1/2}\|a\|_{L^\infty(\Omega)}^{1/2}}{|\mu_n|}
\|h_n\|_{L^2(H^1_a)}.
\end{align*}
This proves \eqref{estimhn} and therefore concludes the proof of Lemma \ref{lemtech}.
\end{proof}
\subsection{Proof of Lemma \ref{convaunabla}}

We now turn to the proof of Lemma \ref{convaunabla}.\newline

\begin{proof}
From \eqref{relwn}, write, 
\[
(1-e^{-i\mu_n s} )w_n = \eta_n + (\mathcal{K}_{\mu_n} h_n)(s) 
\]
Thus, we have, 
\[
\Big( (1-e^{-i\mu_n s} )w_n, (1-e^{i\mu_n s} )w_n \Big)_{L^2(H^1_a)}=\Big(  \eta_n + (\mathcal{K}_{\mu_n} h_n)(s)  , (1-e^{i\mu_n s} )w_n \Big)_{L^2(H^1_a)}
\]
On one hand, we have
\begin{align}
\Big( (1-e^{-i\mu_n s} )w_n, (1-e^{i\mu_n s} )w_n \Big)_{L^2(H^1_a)} & = \int_0^\infty g(s) |1-e^{-i\mu_n s}|^2 \int_\Omega a(x) | \nabla w_n(x) |^2 dx ds \nonumber \\
& \geq C_{\mu_n}   \int_\Omega a(x) | \nabla w_n(x) |^2 dx  \nonumber \\
& = C_{\mu_n}\|\sqrt{a}\nabla w_n \|_{L^2(\Omega)}^2 \label{estimwwnorm}
\end{align}
where $C_{\mu_n}>0$ is defined by 
\[
C_{\mu_n} :=  \int_0^\infty g(s) |1-e^{-i\mu_n s}|^2 ds= 4 \int_0^\infty g(s) \sin^2\left( \dfrac{\mu_n s }{2} \right)  ds. 
\]
We use here the assumption
$\mu_n \rightarrow \mu^* \in (\R \cup \{\pm \infty \} )\setminus \{0\} $ to obtain a positive lower bound on $C_{\mu_n}$ for $n$ large. First, 
\[
C_{\mu_n}=4 \int_0^\infty g(s) \sin^2\left( \dfrac{\mu_n s }{2} \right)  ds = 2 \int_0^\infty g(s) (1- \cos\left( \mu_n s  \right))  ds \rightarrow 2 \int_0^\infty g(s) ds = 2g_0
\]
if $|\mu_n | \rightarrow \infty$ by Riemann-Lebesgue. Thus, assume $\mu_n \rightarrow \mu^* $ with $\mu^*\neq 0$. In this case, for $n$ sufficiently large, we have $|\mu_n|>|\mu^*|/2$. Thus, for $L=4/|\mu^*|>2/|\mu_n|$, we have, 
\begin{align}
C_{\mu_n} = 4 \int_0^\infty g(s) \sin^2\left( \dfrac{\mu_n s }{2} \right)  ds \nonumber 
 & \geq 4 g(L) \int_0^{L}  \sin^2\left( \dfrac{\mu_n s }{2} \right)  ds  \nonumber \\
  & \geq  2g(L)  \left( L -  \frac{\sin ( \mu_n L  )}{\mu_n}  \right) \nonumber \\
& \geq  2g(L)  \left( L -  \frac{1}{|\mu_n|}  \right) \nonumber \\
& \geq  2g(L)  \left( L -  \frac{2}{|\mu^*|}  \right)  \nonumber \\
& \geq   g\left(\dfrac{4}{|\mu^*|}\right)  \left(  \frac{4}{|\mu^*|}  \right). \nonumber
\end{align}
Denote $C^*>0$ the uniform bound $C_{\mu_n}\geq C^*$ for $n$ large and $\mu_n \rightarrow \mu^* \in (\R \cup \{\pm \infty \} )\setminus \{0\} $.\newline

On the other hand, 
\[
\Big|\Big(  \eta_n + (\mathcal{K}_{\mu_n} h_n)(s)  , (1-e^{i\mu_n s} )w_n \Big)_{L^2(H^1_a)}\Big| \leq \Big\| \eta_n + (\mathcal{K}_{\mu_n} h_n)(s)  \Big\|_{L^2(H^1_a)} \Big\|(1-e^{i\mu_n s} )w_n \Big\|_{L^2(H^1_a)}
\]
But, 
\begin{align*}
\|(1-e^{i\mu_n s} )w_n \|_{L^2(H^1_a)}^2 &= \int_0^\infty g(s) \big| 1-e^{i\mu_n s} \big|^2 \int_\Omega a(x) | \nabla w_n(x) |^2 dx ds\\
 & \leq 4 g_0 \|\sqrt{a}\nabla w_n \|_{L^2(\Omega;\CC^d)}^2. 
\end{align*}
Combining this computation with \eqref{estimwwnorm}, we therefore conclude 
\[
\|\sqrt{a}\nabla w_n \|_{L^2(\Omega;\CC^d)} \leq \dfrac{2\sqrt{g_0}}{C_{\mu_n}} \big(\|\eta_n \|_{L^2(H^1_a)} +\| \mathcal{K}_{\mu_n} h_n  \|_{L^2(H^1_a)}\big).
\]
In turn, we deduce for $n$ large and $\mu^* \in (\R \cup \{\pm \infty \} )\setminus \{0\} $
\[
\|\sqrt{a}\nabla w_n \|_{L^2(\Omega;\CC^d)} \leq \dfrac{2 \sqrt{g_0}}{C^*} \big(\|\eta_n \|_{L^2(H^1_a)} +\| \mathcal{K}_{\mu_n} h_n  \|_{L^2(H^1_a)}\big),
\]
for $C^*>0$ independent of $n$. From the observation of $\|\eta_n\|_{L^2(H^1_a)}\rightarrow 0$ given by \eqref{convetaforte}, and \eqref{estimKhn}, 
\begin{equation}\label{wasqrt}
\lim_{n\rightarrow \infty} \|\sqrt{a}\nabla w_n \|_{L^2(\Omega;\CC^d)} = 0.
\end{equation}
 We now turn to \eqref{Jnconv}. By definition, 
\[
J_n = H_n - \dfrac{\nabla f_n }{i\mu_n},
\]
and from the convergence of $ f_n$ in $H^1_0(\Omega)$, we reduce the analysis to 
\[
\lim_{n\rightarrow \infty} |\mu_n| \| H_n\|_{L^2(\Omega;\CC^d)}=0. 
\]
By definition,
\[
H_n(x)= \int_0^\infty g(s)a(x) \nabla (\eta_n -u_n ) ds. 
\]
Using \eqref{relwn}, we have, 
\begin{align}
H_n(x)&=-\int_0^\infty g(s) a(x) \nabla \left(e^{-i\mu_n s}u_n(x) - \left(\dfrac{1-e^{-i\mu_n s}}{i\mu_n}\right)f_n(x) \right)  ds \nonumber \\
& \quad - \int_0^\infty g(s) 
\int_0^s e^{i\mu_n (\sigma-s) } a(x) \nabla h_n(x,\sigma) d\sigma ds \nonumber  \\
& =-\int_0^\infty g(s)e^{-i\mu_n s} a(x) \nabla w_n(x)  ds \nonumber  \\
& \quad +\int_0^\infty g(s) a(x)  \dfrac{ \nabla f_n(x)}{i\mu_n} ds \nonumber  \\
& \quad - \int_0^\infty g(s) 
\int_0^s e^{i\mu_n (\sigma-s) } a(x) \nabla h_n(x,\sigma) d\sigma ds \nonumber  \\
& =- \hat{g}(\mu_n )a(x) \nabla w_n(x) +g_0 a(x) \dfrac{\nabla  f_n(x)}{i\mu_n} \nonumber  \\
& \quad - \int_0^\infty g(s) 
\int_0^s e^{i\mu_n (\sigma-s) } a(x) \nabla h_n(x,\sigma) d\sigma ds
\label{eq:Hn-decomposition}
\end{align}
Multiplying \eqref{eq:Hn-decomposition} by $|\mu_n|$ and using the triangle inequality, we obtain
\begin{align}
 |\mu_n|\lVert H_n\rVert_{L^2(\Omega;\mathbb C^d)}
 &\leq
 |\mu_n \hat g(\mu_n)| \, 
 \lVert a\nabla w_n\rVert_{L^2(\Omega;\mathbb C^d)} \nonumber 
 +g_0\lVert a\nabla  f_n\rVert_{L^2(\Omega;\mathbb C^d)} \nonumber \\
 &\quad
 +|\mu_n|\left\lVert
 \int_0^\infty g(s)\int_0^s
 e^{i\mu_n(\sigma-s)}a(x)\nabla h_n(x,\sigma)
 \,d\sigma\,ds
 \right\rVert_{L^2(\Omega;\mathbb C^d)}
 \label{eq:Hn-bound}
\end{align}
Therefore,
\begin{align*}
 \lVert a\nabla w_n\rVert_{L^2(\Omega;\mathbb C^d)}^2
 =\int_\Omega a(x)^2|\nabla w_n(x)|^2\,dx \leq
 \lVert a\rVert_{L^\infty(\Omega)}
 \lVert\sqrt a\nabla w_n\rVert_{L^2(\Omega;\mathbb C^d)}^2.
\end{align*}
Consequently, estimating $\hat g(\mu_n)$ with Lemma \ref{lemtech} and \eqref{wasqrt}
give
\begin{align*}
 |\mu_n \hat g(\mu_n)| \,
 \lVert a\nabla w_n\rVert_{L^2(\Omega;\mathbb C^d)}
 &\leq
 2g(0)\lVert a\rVert_{L^\infty(\Omega)}^{1/2}
 \lVert\sqrt a\nabla w_n\rVert_{L^2(\Omega;\mathbb C^d)} \rightarrow0.
\end{align*}

For the second term in \eqref{eq:Hn-bound},
\[
 g_0\lVert a\nabla f_n\rVert_{L^2(\Omega;\mathbb C^d)}
 \leq
 g_0\lVert a\rVert_{L^\infty(\Omega)}
 \lVert\nabla f_n\rVert_{L^2(\Omega;\mathbb C^d)}
 \rightarrow0,
\]
since $f_n\rightarrow 0$ in $H_0^1(\Omega)$. Finally, Lemma \ref{lemtech} yields
\begin{align*}
 &|\mu_n|\left\lVert
 \int_0^\infty g(s)\int_0^s
 e^{i\mu_n(\sigma-s)}a(x)\nabla h_n(x,\sigma)
 \,d\sigma\,ds
 \right\rVert_{L^2(\Omega;\mathbb C^d)} \leq
 C\lVert h_n\rVert_{L^2(H^1_a)}
 \rightarrow0,
\end{align*}
since $h_n\to0$ in $L^2(H^1_a)$.  We conclude from
\eqref{eq:Hn-bound},  
\[
 \lim_{n\to\infty}
 |\mu_n|\lVert H_n\rVert_{L^2(\Omega;\mathbb C^d)}=0.
\]
Thus, we have proved, 
\begin{align*}
 |\mu_n|\lVert J_n\rVert_{L^2(\Omega;\mathbb C^d)}
 &\leq
 |\mu_n|\lVert H_n\rVert_{L^2(\Omega;\mathbb C^d)}
 +\lVert\nabla f_n\rVert_{L^2(\Omega;\mathbb C^d)}
 \longrightarrow0.
\end{align*}
This proves \eqref{Jnconv}.
\end{proof}

\bibliographystyle{plain}
\bibliography{biblio}

\begin{thebibliography}{10}

\bibitem{AugnerJacob2014}
B.~Augner and B.~Jacob.
\newblock Stability and stabilization of infinite-dimensional linear
  port-{Hamiltonian} systems.
\newblock {\em Evolution Equations and Control Theory}, 3(2):207--229, 2014.

\bibitem{BLR}
C.~Bardos, G.~Lebeau, and J.~Rauch.
\newblock Sharp sufficient conditions for the observation, control, and
  stabilization of waves from the boundary.
\newblock {\em SIAM J. Control Optim.}, 30(5):1024--1065, 1992.

\bibitem{BurqNontrapping}
N.~Burq.
\newblock Semi-classical estimates for the resolvent in nontrapping geometries.
\newblock {\em Int. Math. Res. Not.}, 2002(5):221--241, 2002.

\bibitem{BDLRI}
N.~Burq, B.~Dehman, and J.~Le~Rousseau.
\newblock Measure and continuous vector field at a boundary {I}: propagation
  equations and wave observability, 2024.
\newblock arXiv:2407.02255.

\bibitem{BDLRII}
N.~Burq, B.~Dehman, and J.~Le~Rousseau.
\newblock Measure and continuous vector field at a boundary {II}: geodesics and
  support propagation, 2024.
\newblock arXiv:2407.02259.

\bibitem{CoursBG}
N.~Burq and P.~G{\'e}rard.
\newblock Contrôle optimal des équations aux dérivées partielles.
\newblock 2001.

\bibitem{Cavalcanti2}
M.~M. Cavalcanti, V.~N. Domingos~Cavalcanti, M.~A. Jorge~Silva, and A.~Y.
  de~Souza~Franco.
\newblock Exponential stability for the wave model with localized memory in a
  past history framework.
\newblock {\em JDE}, 264(11):6535--6584, 2018.

\bibitem{ContiDellOroPata2013}
M.~Conti, F.~Dell'Oro, and V.~Pata.
\newblock {Timoshenko} systems with fading memory.
\newblock {\em Dynamics of Partial Differential Equations}, 10(4):367--377,
  2013.

\bibitem{Dafermos}
C.~M. Dafermos.
\newblock Asymptotic stability in viscoelasticity.
\newblock {\em Arch. Ration. Mech. Anal.}, 37:297--308, 1970.

\bibitem{EN}
K.-J. Engel and R.~Nagel.
\newblock {\em One-parameter semigroups for linear evolution equations}, volume
  194 of {\em Grad. Texts Math.}
\newblock Berlin: Springer, 2000.

\bibitem{Jorge}
J.~C.~O. Faria, M.~A. Jorge~Silva, and A.~Y. Souza~Franco.
\newblock A general stability result for the semilinear viscoelastic wave model
  under localized effects.
\newblock {\em Nonlinear Anal., Real World Appl.}, 56:34, 2020.

\bibitem{FariaOkawaSouzaFranco2026}
J.~C.~O. Faria, C.~A. Okawa, and A.~Y. Souza~Franco.
\newblock Asymptotic behavior of a semilinear wave equation with both localized
  memory and delay effects.
\newblock {\em Asymptotic Analysis}, pages 1--27, 2026.

\bibitem{FariaSouzaFranco2023}
J.~C.~O. Faria and A.~Y. Souza~Franco.
\newblock Well-posedness and exponential stability for a {Klein--Gordon} system
  with locally distributed viscoelastic dampings in a past-history framework.
\newblock {\em Journal of Differential Equations}, 346:108--144, 2023.

\bibitem{timedependentwave}
J.~Le~Rousseau, G.~Lebeau, P.~Terpolilli, and E.~Tr{\'e}lat.
\newblock Geometric control condition for the wave equation with a
  time-dependent observation domain.
\newblock {\em Anal. PDE}, 10(4):983--1015, 2017.

\bibitem{LiuRao2004}
K.~Liu and B.~Rao.
\newblock Stabilit{\'e} exponentielle des {\'e}quations des ondes avec
  amortissement local de {Kelvin--Voigt}.
\newblock {\em Comptes Rendus. Math{\'e}matique}, 339(11):769--774, 2004.

\bibitem{LiuZheng1999}
Z.~Liu and S.~Zheng.
\newblock {\em Semigroups Associated with Dissipative Systems}, volume 398 of
  {\em Chapman \& Hall/CRC Research Notes in Mathematics Series}.
\newblock CRC Press, Boca Raton, FL, 1999.

\bibitem{MelroseSjostrand1978}
R.~B. Melrose and J.~Sj{\"o}strand.
\newblock Singularities of boundary value problems. {I}.
\newblock {\em Communications on Pure and Applied Mathematics}, 31(5):593--617,
  1978.

\bibitem{MelroseSjostrand1982}
R.~B. Melrose and J.~Sj{\"o}strand.
\newblock Singularities of boundary value problems. {II}.
\newblock {\em Communications on Pure and Applied Mathematics}, 35(2):129--168,
  1982.

\bibitem{Piupiu}
Q.~Peng and Y.~Liu.
\newblock Exponential stability for an infinite memory wave equation with
  frictional damping and logarithmic nonlinear terms.
\newblock {\em Nonlinear Anal., Real World Appl.}, 88:16, 2026.

\bibitem{Pruss1984}
J.~Pr{\"u}ss.
\newblock On the spectrum of {$C_0$}-semigroups.
\newblock {\em Transactions of the American Mathematical Society},
  284(2):847--857, 1984.

\bibitem{Wyler1994}
A.~Wyler.
\newblock Stability of wave equations with dissipative boundary conditions in a
  bounded domain.
\newblock {\em Asymptotic Analysis}, 8(4):379--407, 1994.

\bibitem{ZworskiBook}
M.~Zworski.
\newblock {\em Semiclassical analysis}, volume 138 of {\em Grad. Stud. Math.}
\newblock Providence, RI: American Mathematical Society (AMS), 2012.

\end{thebibliography}

\end{document}